\documentclass[a4paper, 10.5pt]{amsart}

\usepackage[hmargin=2.7cm, vmargin=3cm, bindingoffset=0.5cm]{geometry}
\usepackage{amsthm,amssymb,amsmath,amsfonts,mathrsfs,amscd}
\usepackage[latin1]{inputenc}
\usepackage[all]{xy}
\usepackage{latexsym}
\usepackage{longtable}
\usepackage{url}
\usepackage{stmaryrd}
\usepackage{tikz-cd}
\usepackage{bbm}
\usepackage{hyperref}
\usepackage{geometry}
\usepackage{marginnote}
\usepackage{enumitem}
\usepackage{color}

\newtheorem{thm}{Theorem}[section]
\newtheorem*{theo}{Theorem}
\newtheorem{conj}[thm]{Conjecture}
\newtheorem{lemma}[thm]{Lemma}
\newtheorem{prop}[thm]{Proposition}
\newtheorem{corollary}[thm]{Corollary}
\theoremstyle{definition}
\newtheorem{definition}[thm]{Definition}

\theoremstyle{remark}
\newtheorem{rmk}[thm]{Remark}

\def\Q{\mathbb{Q}}

\def\Qp{\mathbb{Q}_p}
\def\Zp{\mathbb{Z}_p}
\def\Fp{\mathbb{F}_p}
\def\bQp{\overline{\mathbb{Q}}_p}

\def\GQp{G_{\Qp}}
\def\IQp{I_{\Qp}}

\def\1{\mathbbm{1}}
\def\GL{\operatorname{GL}}
\def\SL{\operatorname{SL}}
\def\End{\operatorname{End}}
\def\mSpec{\operatorname{m-Spec}}
\def\Spec{\operatorname{Spec}}
\def\WD{\operatorname{WD}}
\def\LL{\operatorname{LL}}
\def\Hom{\operatorname{Hom}}
\def\Homc{\operatorname{Hom}^{\operatorname{cont}}}
\def\det{\operatorname{det}}
\def\tr{\operatorname{tr}}
\def\Sym{\operatorname{Sym}}
\def\Ext{\operatorname{Ext}}
\def\cInd{\operatorname{c-Ind}^G_K}
\def\ctimes{\hat{\otimes}}
\def\Ker{\operatorname{Ker}}

\def\BS{\operatorname{BS}}
\def\pialg{\pi_{\operatorname{alg}}}
\def\pism{\pi_{\operatorname{sm}}}
\def\lalg{\operatorname{l.alg}}
\def\Art{\operatorname{Art}_{\Qp}}
\def\Fil{\operatorname{Fil}}
\def\Sp{\operatorname{Sp}}
\def\Ind{\operatorname{Ind}}
\def\ind{\operatorname{ind}}

\def\Ord{\operatorname{Ord}}
\def\ROrd{\operatorname{R^1Ord}}
\def\st{\tilde{\operatorname{st}}}

\def\rbar{\bar{r}}
\def\Runiv{R_{\rbar}^{\Box}}
\def\Runivz{R_{\rbar}^{\Box, \psi}}
\def\Rpsi{R_{\bar{\psi}}}
\def\runiv{r^{\Box}}
\def\psiuniv{{\psi}^{\operatorname{univ}}}
\def\Rss{R_{\rbar}^{\Box, \psi}(\lambda, \tau)}
\def\Rcr{R_{\rbar}^{\Box, \psi ,cr}(\lambda, \tau)}
\def\Rinf{R_{\infty}}
\def\Rinfz{R_{\infty}^{\psi}}
\def\Minf{M_{\infty}}
\def\rinf{r_{\infty}}
\def\Piinf{\Pi_{\infty}}
\def\Piinfz{\Pi_{\infty}^{\psi}}
\def\Minfz{M_{\infty}^{\psi}}
\def\Modsm{\operatorname{Mod}^{\operatorname{sm}}_G}
\def\Modsmc{\operatorname{Mod}^{\operatorname{sm}}_{G, \psi}}
\def\Modlfin{\operatorname{Mod}^{\operatorname{l.fin}}_{G}}
\def\Modlfinc{\operatorname{Mod}^{\operatorname{l.fin}}_{G, \psi}}
\def\Modladmc{\operatorname{Mod}^{\operatorname{l.adm}}_{G, \psi}}
\def\Modpro{\operatorname{Mod}^{\operatorname{pro}}}
\def\Irr{\operatorname{Irr}_{G, \psi}(k)}
\def\Ban{\operatorname{Ban}^{\operatorname{adm}}_{G, \psi}}
\def\bracketGQp{\llbracket \GQp \rrbracket}
\def\bracketK{\llbracket K \rrbracket}
\def\bracketx{\llbracket x_1, ..., x_d \rrbracket}
\def\Dcris{D_{\operatorname{cris}}}
\def\DdR{D_{\operatorname{dR}}}
\def\Gal{\operatorname{Gal}}
\def\bsigma{\bar{\sigma}}
\def\cV{\check{\mathbf{V}}}

\def\O{\mathcal{O}}
\def\C{\mathfrak{C}}
\def\I{\mathcal{I}}
\def\D{\mathfrak{D}}
\def\T{\mathfrak{T}}
\def\p{\mathfrak{p}}
\def\B{\mathfrak{B}}
\def\m{\mathfrak{m}}
\def\q{\mathfrak{q}}
\def\fI{\mathfrak{I}}
\def\H{\mathcal{H}}
\def\K{\mathcal{K}}
\def\P{\mathcal{P}}
\def\G{\mathcal{G}}
\def\cS{\mathcal{S}}
\def\cC{\mathcal{C}}
\def\cZ{\mathcal{Z}}
\def\cT{\mathcal{T}}
\def\V{\mathbf{V}}

\title{On the automorphy of 2-dimensional potentially semi-stable deformation rings of $\GQp$}
\author{Shen-Ning Tung}

\address{Fakult\"at f\"ur Mathematik\\
  Universit\"at Duisburg-Essen\\
  45127 Essen, Germany}
\email{shen-ning.tung@stud.uni-due.de}

\begin{document}

\begin{abstract}
Using $p$-adic local Langlands correspondence for $\GL_2(\Qp)$, we prove that the support of patched modules $\Minf(\sigma)[1/p]$ constructed in \cite{MR3529394} meet every irreducible component of the potentially semi-stable deformation ring $\Runiv(\sigma)[1/p]$. This gives a new proof of the Breuil-M\'{e}zard conjecture for 2-dimensional representations of the absolute Galois group of $\Qp$ when $p > 2$, which is new in the case $p=3$ and $\rbar$ a twist of an extension of the trivial character by the mod $p$ cyclotomic character. As a consequence, a local restriction in the proof of Fontaine-Mazur conjecture in \cite{MR2505297} is removed.
\end{abstract}

\maketitle


\section*{Introduction}
Let $F / \Qp$ be a finite extension and $G_F$ be the absolutely Galois group of $F$. Fix a finite extension $E / \Qp$ with ring of integers $\O$ and residue field $k$, a continuous representation $\rbar : G_F \rightarrow \GL_n(k)$ and a continuous character $\psi: G_F \rightarrow \O^{\times}$. Denote $\varepsilon$ the $p$-adic cyclotomic character and $\Runiv$ (resp. $\Runivz$) the universal framed deformation ring (resp. universal framed deformation ring with a fixed determinant $\psi \varepsilon$) of $\rbar$. Under the assumption that $p$ does not divide $2n$, \cite{MR3529394} constructs an $\Rinf[G]$-module $\Minf$ by applying Taylor-Wiles-Kisin patching method to algebraic automorphic forms on some definite unitary group, where $\Rinf$ is a complete local noetherian $\Runiv$-algebra with residue field $k$ and $G = \GL_n(F)$. If $y \in \mSpec \Rinf[1/p]$, then
\begin{gather*}
\Pi_y := \Homc_{\O}\big(\Minf \otimes_{\Rinf, y} \kappa(y), E \big)
\end{gather*}
with $\kappa(y)$ the residue field at $y$, is an admissible unitary $E$-Banach space representation of $G$. The composition $\Runiv \rightarrow \Rinf \xrightarrow{y} \kappa(y)$ defines an $\kappa(y)$-valued point $x \in \mSpec \Runiv[1/p]$ and thus a continuous Galois representation $r_x : G_F \rightarrow \GL_n \big(\kappa(y) \big)$. It is expected that the Banach space representation $\Pi_y$ depends only on $x$ and that it should be related to $r_x$ by the hypothetical $p$-adic local Langlands correspondence; see \cite{MR3529394} and \cite{MR3732208} for a detailed discussion.

In this paper, we specialize this construction to the case $F = \Qp$ and $n=2$, in which the $p$-adic local Langlands correspondence is known. The goal of this paper is to prove that every irreducible component of a potentially semi-stable deformation ring is automorphic. This amounts to showing that if $r_x$ is potentially semi-stable with distinct Hodge-Tate weights, then (a subspace of) locally algebraic vectors in $\Pi_y$ can be related to $\WD(r_x)$ via the classical local Langlands correspondence, where $\WD(r_x)$ is the Weil-Deligne representation associated to $r_x$ by Fontaine.

One of the ingredients is a result of Emerton and Pa\v sk\=unas in \cite{MR4077579}, which shows that the action of $\Rinf$ on $\Minf$ is faithful. The statement can also be deduced from the work of  Hellmann and Schraen \cite{MR3542488}. Note that this does not imply that $\Pi_y \neq 0$ since $\Minf$ is not finitely generated over $\Rinf$. We overcome this problem by applying Colmez's Montreal functor $\cV$ to $\Minfz$ (a fixed determinant quotient of $\Minf$) and showing that $\cV(\Minfz)$ is a finitely generated $\Rinfz$-module, where $\Rinfz$ is a complete local noetherian $\Runivz$-algebra. This finiteness result is a key idea in this paper. While our paper is written, we were notified that Lue Pan also obtains a similar result independently in \cite{2019arXiv190107166P}. However, we would like to emphasize that our results are most interesting in the cases not covered by \cite{MR3150248}, whose results Pan uses in an essential way.

Our strategy runs as follows: we first show that $\cV(\Minfz)$ is finitely generated over $\Rinfz \bracketGQp$ and the action of $\Rinfz \bracketGQp$ on $\cV(\Minfz)$ factors through the 2-sided ideal $J$ generated by Cayley-Hamilton relations, which imply that $\cV(\Minfz)$ is a finitely generated module over $\Rinfz$. Using the result of Emerton and Pa\v sk\=unas mentioned above, we show that the action of $\Rinfz$ on $\cV(\Minfz)$ is faithful. Since $\cV(\Minfz)$ is a finitely generated $\Rinfz$-module, this implies that the specialization of $\cV(\Minfz)$ at any $y \in \mSpec \Rinfz[1/p]$ is non-zero, which in turn implies that $\Pi_y$ is nonzero. Combining these, we prove that every irreducible component of a potentially semi-stable deformation ring is automorphic if it contains a point whose associated Galois representation $r_x$ is irreducible. So we only have to handle components in the reducible (thus ordinary) locus, which is known to be automorphic by \cite{MR3072811} except for the case which admits reducible semi-stable non-crystalline components. This gives a new proof of the Breuil-M\'ezard conjecture outside this particular case by the formalism of \cite{MR2505297, MR3292675, MR3134019, MR3306557}. Fortunately, the conjecture for this case can be checked directly by computing Hilbert-Samuel multiplicities, thus we prove the conjecture. This is new in the case $p=3$ and $\rbar$ a twist of the trivial character by the mod $p$ cyclotomic character. 

As a consequence, the Fontaine-Mazur conjecture without the local restriction follows immediately from the original proof of \cite{MR2505297}.

\begin{theo}
Let $p>2$, $S$ a finite set of primes containing $\{p, \infty \}$, $G_{\Q,S}$ the Galois group of the maximal extension of $\Q$ unramified outside $S$, and $G_{\Qp} \subset G_{\Q,S}$ a decomposition group at $p$. Let $\rho : G_{\Q, S} \rightarrow \GL_2(\O)$ be a continuous irreducible odd representation. Suppose that
\begin{enumerate}
\item $\bar{\rho} \vert_{\Q(\zeta_p)}$ is absolutely irreducible.
\item $\rho \vert_{G_{\Qp}}$ is potentially semi-stable with distinct Hodge-Tate weights.
\end{enumerate}
Then (up to twist) $\rho$ comes from a cuspidal eigenforms.
\end{theo}

A similar strategy works in the case when $p=2$ and is handled in \cite{Tung2020}.


The paper is organized as follows. In \S \ref{Section:BM}, we state the Breuil-M\'ezard conjecture for 2-dimensional Galois representations of $\GQp$. In \S \ref{Section:rep}, we recall the $p$-adic representation theory for $\GL_2(\Qp)$ and the Colmez's Montreal functor. In \S \ref{Section:patched}, we define patched modules in our setting and show some of their properties after applying Colmez's functor. In \S \ref{Sectiuon:main}, we prove the automorphy of potentially semi-stable deformation rings and explain how such a result can be applied to the Breuil-M\'ezard and Fontaine-Mazur conjectures.

\section*{Acknowledgments}
Part of this paper is my Ph.D. thesis at the University of Duisburg-Essen. I want to thank my advisor Vytautas Pa\v sk\=unas for suggesting this problem and his constant encouragement to work on this topic. I also thanks Shu Sasaki for his help on the modularity lifting theorem. I am indebted to James Newton, Florian Herzig, and the anonymous referees for pointing out several blunders and for their valuable comments. This research was funded in part by the DFG, SFB/TR 45 "Periods, moduli spaces and arithmetic of algebraic varieties".

\section*{Notation}
\begin{itemize}
\item $p$ is an odd prime number.
\item $E / \Qp$ is a sufficiently large finite extension with ring of integers $\O$, uniformizer $\varpi$ and residue field $k$.
\item For a number field $F$, the completion at a place $v$ is written as $F_v$.
\item For a local or global field $L$, $G_L = \Gal(\bar{L} / L)$. The inertia subgroup for the local field is written as $I_L$.
\item $\varepsilon : \GQp \rightarrow \Zp^{\times}$ is the $p$-adic cyclotomic character, whose Hodge-Tate weight is defined to be 1. 
\item $\omega : \GQp \rightarrow \Fp^{\times}$ is the mod $p$ cyclotomic character, and $\1 : \GQp \rightarrow \Fp^{\times}$ is the trivial character. We also denote other trivial representations by $\1$ if no confusion arises.
\item Normalize the local class field theory $\Art : \Qp^{\times} \rightarrow \GQp^{ab}$ so that $p$ maps to a geometric Frobenius. Then a character of $\GQp$ will also be regarded as a character of $\Qp^{\times}$.
\item For a ring $R$, $\mSpec R$ denotes the set of maximal ideals of $R$.
\item For $R$ a Noetherian ring and $M$ a finite $R$-module of dimension at most $d$, let $\ell_{R_{\p}}(M_{\p})$ denote the length of the $R_{\p}$-module $M_{\p}$, and let $\cZ_d(M) := \sum_{\p} \ell_{R_{\p}} (M_{\p}) \p$ be a cycle, where the sum is taken over all $\p \in \Spec R$ such that $\dim R / \p =d$. If the support of $M$ is equidimensional of some finite dimension $d$, then we write simply $\cZ(M) := \cZ_d(M)$. 
\item For $R$ a Noetherian local ring with maximal ideal $\m$ and $M$ a finite $R$-module, and for an $\m$-primary ideal $\q$ of R, let $e_{\q}(R, M)$ denote the Hilbert-Samuel multiplicity of $M$ with respect to $\q$. We abbreviate $e_{\m}(R, R) = e(R)$.
\item If $A$ is a topological $\O$-module, we write $A^{\vee} := \Homc_{\O}(A, E/\O)$ for the Pontryagin dual of $A$. If $A$ is a pseudocompact $\O$-torsion free $\O$-module, we write $A^d := \Homc_{\O}(A, \O)$ for its Schikhof dual.
\item We write $G = \GL_2(\Qp)$ and $K = \GL_2(\Zp)$, and let $Z = Z(G) \cong \Qp^{\times}$ denote the center of $G$.
\end{itemize}
\section{The Breuil-M\'{e}zard conjecture} \label{Section:BM}
Consider the following data:
\begin{itemize}
\item a pair of integers $\lambda = (a, b)$ with $a > b$,
\item a representation $\tau : I_{\Qp} \rightarrow \GL_2(E)$ with an open kernel,
\item a continuous character $\psi : \GQp \rightarrow \O^{\times}$ such that $\psi \vert_{\IQp} = \varepsilon^{a+b-1} \det \tau$.
\end{itemize}
We call such a triple $(\lambda, \tau, \psi)$ a $p$-adic Hodge type. We say a 2-dimensional continuous representation $r : \GQp \rightarrow \GL_2(E)$ is of type $(\lambda, \tau, \psi)$ if $r$ is potentially semi-stable (= de Rham) such that its Hodge-Tate weights are $(a, b)$, $\WD(r \vert_{I_{\Qp}}) \cong \tau$ and $\det r = \psi \varepsilon$. Here $\WD(r)$ is the Weil-Deligne representation associated to $r$ by Fontaine.

By a result of Henniart in the appendix of \cite{MR1944572}, there is a unique finite-dimensional smooth irreducible $\bQp$-representation $\sigma(\tau)$ (resp. $\sigma^{cr}(\tau)$) of $K$, such that for any infinite-dimensional smooth absolutely irreducible representation $\pi$ of $G$ and the associated Weil-Deligne representation $\LL(\pi)$ attached to $\pi$ via the classical local Langlands correspondence, we have $\Hom_K(\sigma(\tau), \pi) \neq 0$ (resp. $\Hom_K(\sigma^{cr}(\tau), \pi) \neq 0$) if and only if $\LL(\pi) \vert_{I_{\Qp}} \cong \tau$ (resp. $\LL(\pi) \vert_{I_{\Qp}} \cong \tau$ and the monodromy operator $N$ is trivial). We remark that $\sigma(\tau)$ and $\sigma^{cr}(\tau)$ differ only when $\tau \cong \eta \oplus \eta$ is scalar, in which case 
\begin{gather*}
\sigma(\tau) \cong \st \otimes \eta \circ \det, \quad \sigma^{cr}(\tau) \cong \eta \circ \det
\end{gather*}
where $\st$ is the inflation to $\GL_2(\Zp)$ of the Steinberg representation of $\GL_2(\Fp)$.

Enlarging $E$ if needed, we may assume $\sigma(\tau)$ is defined over $E$. We write $\sigma(\lambda) = \Sym^{a-b-1}E^2 \otimes \det^b$ and $\sigma(\lambda, \tau)= \sigma(\lambda) \otimes \sigma(\tau)$. Since $\sigma(\lambda, \tau)$ is a finite-dimensional $E$-vector space, $K = \GL_2(\Zp)$ is compact, and the action of $K$ on $\sigma(\lambda, \tau)$ is continuous, there is a $K$-stable $\O$-lattice $\sigma^{\circ}(\lambda, \tau)$ in $\sigma(\lambda, \tau)$. Then $\sigma^{\circ}(\lambda, \tau) / (\varpi)$ is a smooth finite length $k$-representation of $K$, we will denote by $\overline{\sigma(\lambda, \tau)}$ its semi-simplification. One may show that $\overline{\sigma(\lambda, \tau)}$ does not depends on the choice of a lattice. For each smooth irreducible $k$-representation $\bsigma$ of $K$, we let $m_{\bsigma}(\lambda, \tau)$ be the multiplicity with which $\bsigma$ occurs in $\overline{\sigma(\lambda, \tau)}$. Similarly, we let $\sigma^{cr}(\lambda, \tau)= \sigma(\lambda) \otimes \sigma^{cr}(\tau)$, and we let $m_{\bsigma}^{cr}(\lambda, \tau)$ be the multiplicity with which $\bsigma$ occurs in $\overline{\sigma^{cr}(\lambda, \tau)}$.

Let $\rbar:\GQp \rightarrow \GL_2(k)$ be a continuous representation. We will write $\Runiv$ (resp. $\Runivz$) for the universal framed deformation ring of $\rbar$ (resp. universal framed deformation ring of $\rbar$ with determinant $\psi \varepsilon$) and $\runiv:\GQp \rightarrow \GL_2(\Runiv)$ for the universal framed deformation. If $x \in \mSpec \Runiv[1/p]$, then the residue field $\kappa(x)$ is a finite extension of $E$. Let $\O_{\kappa(x)}$ be the ring of integers in $\kappa(x)$. By specializing the universal framed deformation at $x$, we obtain a continuous representation $r_x:\GQp \rightarrow \GL_2(\O_{\kappa(x)})$, which reduces to $\rbar$ modulo the maximal ideal of $\O_{\kappa(x)}$.

Since $\kappa(x)$ is a finite extension of $E$, $r_x$ lies in Fontaine's $p$-adic Hodge theory \cite{MR1293972}. Kisin has shown the existence of a reduced, 4-dimensional (if non-trivial), $p$-torsion free quotient $\Rss$ (resp. $\Rcr$) of $\Runiv$ such that for all $x \in \mSpec \Runiv[1/p]$, $x$ lies in $\mSpec \Rss[1/p]$ (resp. $\Rcr[1/p]$) if and only if $\runiv_{x}$ is potentially semi-stable (resp. potentially crystalline) of type $(\lambda, \tau, \psi)$.

In \cite{MR1944572}, Breuil and M\'ezard made an conjecture relating the Hilbert-Samuel multiplicity of $\Rss / \varpi$ (resp. $\Rcr / \varpi$) with the number $m_{\bsigma}(\lambda, \tau)$ (resp. $m_{\bsigma}^{cr}(\lambda, \tau)$) defined above.

\begin{conj}[Breuil-M\'ezard] \label{conj:Breuil-Mezard}
For each smooth irreducible $k$-representation $\bsigma$ of $K$, there exists an integer $\mu_{\bsigma}(\rbar)$, independent of $\lambda$ and $\tau$, such that for all $p$-adic Hodge types $(\lambda, \tau, \psi)$, we have equalities:
\begin{align*}
e(\Rss / \varpi) = \sum_{\bsigma} m_{\bsigma}(\lambda, \tau) \mu_{\bsigma}(\rbar) \\
e(\Rcr / \varpi) = \sum_{\bsigma} m_{\bsigma}^{cr}(\lambda, \tau) \mu_{\bsigma}(\rbar)
\end{align*}
where the sum is taken over the set of isomorphism classes of smooth irreducible $k$-representations of $K$.
\end{conj}

Conjecture \ref{conj:Breuil-Mezard} was proved by \cite{MR2505297, MR3306557, MR3429471, MR3544298} in all cases if $p \geq 5$ and in the case $\rbar$ is not a twist of an extension of $\1$ by $\omega$ if $p =2, 3$. We prove the conjecture for $p > 2$ in terms of cycles formulated in \cite{MR3134019} by a new global method building on the method of \cite{MR3529394}. The result is new in the case $p = 3$ and $\rbar = (\begin{smallmatrix} \omega & * \\ 0 & \1 \end{smallmatrix}) \otimes \chi$.

\begin{thm} \label{thm:Breuil-Mezard}
Assume $p > 2$. For each smooth irreducible $k$-representation $\bsigma$ of $K$, there exists a 4-dimensional cycle $\cC_{\bsigma}(\rbar)$ of $\Runivz$, independent of $\lambda$ and $\tau$, such that for all $p$-adic Hodge types $(\lambda, \tau, \psi)$, we have equalities of 4-dimensional cycles:
\begin{align*}
\cZ(\Rss / \varpi) = \sum_{\bsigma} m_{\bsigma}(\lambda, \tau) \cC_{\bsigma}(\rbar) \\
\cZ(\Rcr / \varpi) = \sum_{\bsigma} m_{\bsigma}^{cr}(\lambda, \tau) \cC_{\bsigma}(\rbar)
\end{align*}
where the sum is taken over the set of isomorphism classes of smooth  irreducible $k$-representations of $K$.
\end{thm}

\begin{rmk} \label{rmk:cryslift}
If $\tau$ is a 2-dimensional trivial representation of $I_{\Qp}$, then $\sigma^{cr}(\lambda, \tau) = \Sym^{a-b-1} E^2 \otimes \det^b$.	If moreover $1 \leq a-b \leq p$, then $\overline{\sigma^{cr}(\lambda, \tau)} \cong \Sym^{a-b-1} k^2 \otimes \det^b$ is an irreducible representation of $K$ over $k$. Thus in this case $m_{\bsigma} = 0$ unless $\bsigma \cong \Sym^{a-b-1} k^2 \otimes \det^b$, in which case the multiplicity is equal to 1. This observation together with Theorem \ref{thm:Breuil-Mezard} implies that if $\bsigma = \Sym^r k^2 \otimes \det^s$ with $0 \leq r \leq p-1$ and $0 \leq s \leq p-2$ then for
$\lambda = (r+s+1, s)$ and $\tau = \1 \oplus \1$, we have the following equalities
\begin{align*}
\mu_{\bsigma}(\rbar) = e(\Rcr / \varpi); \\
\cC_{\bsigma}(\rbar) = \cZ(\Rcr / \varpi).
\end{align*}
\end{rmk}
\section{Preliminaries on the representation theory of $\GL_2(\Qp)$} \label{Section:rep}
We let $G = \GL_2(\Qp)$, $K = \GL_2(\Zp)$ and $Z \simeq \Qp^\times$ be the center of $G$. Let $B$ be the subgroup of upper triangular matrices in $G$. If $\chi_1$ and $\chi_2$ are characters of $\Qp^{\times}$, then we write $\chi_1 \otimes \chi_2$ for the character of $B$ which maps $(\begin{smallmatrix}a & b \\ 0 & d \end{smallmatrix})$ to $\chi_1(a) \chi_2(d)$.

Let $\Modsm(\O)$ be the category of smooth $G$-representation on $\O$-torsion modules. An object $\pi \in \Modsm(\O)$ is locally finite if for all $v \in \pi$, the $\O[G]$-submodule generated by $v$ is of finite length. Let $\Modlfin(\O)$ to be full subcategory of $\Modsm(\O)$ consisting of all locally finite representations and define $\Modsm(k)$ and $\Modlfin(k)$ in the same way with $\O$ replaced by $k$. Moreover for a continuous character $\psi : Z \rightarrow \O^\times$, adding the subscript $\psi$ in any of the above categories indicates the corresponding full subcategory of $G$-representations with central character $\psi$.

An object $\pi$ of $\Modsm(\O)$ is called admissible if $\pi^H[\varpi^i]$ is a finitely generated $\O$-module for every open compact subgroup $H$ of $G$ and every $i \geq 1$; $\pi$ is called locally admissible if for every $v \in \pi$, the smallest $\O[G]$-submodule of $\pi$ containing $v$ is admissible. Let $\Modladmc(\O)$ be a full subcategory of $\Modsmc(\O)$ consisting of locally admissible representations. In \cite{MR2667882}, Emerton shows that $\Modladmc(\O)$ is an abelian category and $\Modladmc(\O) \cong \Modlfinc(\O)$.

Let $\Modpro_G(\O)$ be the category of compact $\O \bracketK$-modules with an action of $\O[G]$ such that the two actions coincide when restricted to $\O[K]$. This category is anti-equivalent to $\Modsm(\O)$ under the Pontryagin dual $\pi \mapsto \pi^\vee := \Hom_{\O}(\pi, E / \O)$, where the former is equipped with the discrete topology and the latter is equipped with the compact-open topology. Finally let $\C(\O)$ and $\C(k)$ be respectively the full subcategory of $\Modpro_G(\O)$ anti-equivalent to $\Modlfinc(\O)$ and $\Modlfinc(k)$.

Every irreducible object $\pi$ of $\Modsm(\O)$ is killed by $\varpi$, and hence is an object of $\Modsm(k)$. Barthel-Livn\'{e} \cite[Theorem 33]{MR1290194} and Breuil \cite[Theorem 1.6]{MR2018825} have classified the absolutely irreducible smooth representations $\pi$ admitting a central character. They fall into four disjoint classes:
\begin{itemize}
\item characters $\eta \circ \det$;
\item principal series $\Ind_B^G(\chi_1 \otimes \chi_2)$, with $\chi_1 \neq \chi_2$;
\item special series $\Sp \otimes \eta \circ \det$, where $\Sp$ is the Steinberg representation defined by the exact sequence $0 \rightarrow \1 \rightarrow \Ind_B^G \1 \rightarrow \Sp \rightarrow 0$;
\item supersingular $\cInd \sigma / (T, S- s)$, where $\sigma$ is a smooth irreducible $k$-representation of $K$, $T, S \in \End_G(\cInd \sigma)$ are certain Hecke operator defined in \cite{MR1290194, MR2018825}, and $s \in k^{\times}$.
\end{itemize}

Let $\Ban(E)$ be the category of admissible unitary $E$-Banach space representations on which $Z$ acts by $\psi$ (see \cite[\S 3]{MR1900706}). If $\Theta$ is an open bounded $G$-invariant $\O$-lattice in $\Pi \in \Ban(E)$, then the Schikhof dual $\Theta^d$ equipped with the weak topology is an object of $\C(\O)$ (see \cite[Lemmas 4.4, 4.6]{MR3150248}).

\subsection{Blocks}
For $\pi', \pi \in \Modlfinc(k)$, we denote $\Ext^1_{G/Z}(\pi', \pi)$ the Yoneda extension of $\pi'$ by $\pi$ in $\Modlfinc(k)$. Let $\Irr$ be the set of equivalent classes of smooth irreducible $k$-representations of $G$ with central character $\psi$. We say $\pi, \pi' \in \Irr$ is in the same block if there exist $\pi_1, ..., \pi_n \in \Irr$, such that $\pi \cong \pi_1$, $\pi' \cong \pi_n$ and either $\Ext^1_{G, \psi}(\pi_i, \pi_{i+1})$ or $\Ext^1_{G, \psi}(\pi_{i+1}, \pi_{i})$ is nonzero for $1 \leq i \leq n-1$. The classification of blocks can be found in \cite[Corollary 1.2]{MR3306557}. By Proposition 5.34 of \cite{MR3150248}, the category $\Modlfinc(\O)$ decomposes into a direct product of subcategories
\begin{align} \label{equation:blockdecomp}
\Modlfinc(\O) \cong \prod_{\B} \Modlfinc(\O)[\B]
\end{align}
where the product is taken over all the blocks $\B$ and the objects of $\Modlfinc(\O)[\B]$ are representations whose irreducible subquotients lie in $\B$. Let $\C(\O)[\B]$ be the full subcategory of $\C(\O)$ consisting of all $M$ whose irreducible subquotients lie in the dual of $\B$. Thus $\C(\O)[\B]$ is anti-equivalent to $\Modlfinc(\O)[\B]$ under Pontryagin duality and we have another decomposition of categories
\begin{align} \label{equation:Cblockdecomp}
\C(\O) \cong \prod_\B \C(\O)[\B]
\end{align}
coming from (\ref{equation:blockdecomp}).

Let $\T(\O)$ be the full subcategory of $\C(\O)$ whose objects have trivial $\SL_2(\Qp)$-action. It follows from \cite[Lemma 10.25]{MR3150248} that $\T(\O)$ is a thick subcategory of $\C(\O)$ and hence we may consider the quotient category $\D(\O) := \C(\O) / \T(\O)$. Let $\cT: \C(\O) \rightarrow \D(\O)$ be the functor $\cT M = M$ for every object of $\C(\O)$ and $\cT f$ the image of $f: M \rightarrow N$ in $\Hom_{\D(\O)}(\cT M, \cT N)$. It is shown in \cite[\S 10]{MR3150248} that $\D(\O)$ is an abelian category with enough projectives and $\cT$ is an exact functor. We denote $\D(k)$ the full subcategory of $\D(\O)$ consisting of objects killed by $\varpi$.

\begin{rmk}
Note that in \cite[\S 10]{MR3150248}, Pa\v{s}k\={u}nas considers the category $\fI(\O)$ of compact $\O$-modules with trivial $G$-action, which is a thick subcategory of $\C(\O)$ when $\psi = \1$. Since $G / \SL_2(\Qp) Z$ has order prime to $p$ (since $p>2$), $\fI(\O)$ coincides with $\T(\O)$ when $\psi = \1$ and $p>3$ ($\omega \circ \det$ is fixed by $\SL_2(\Qp)$ but not by $\GL_2(\Qp)$).
\end{rmk}

Let $\overline{\chi}: Z \rightarrow k^{\times}$ be a character and $\overline{\psi} = \overline{\chi}^2$. Since $\T(\O)$ is contained in $\C(\O)[\B]$ with $\B = \{\1, \Sp, \Ind^G_B \omega \otimes \omega^{-1} \} \otimes \overline{\chi} \circ \det$ when $p \geq 5$ and $\B = \{\1, \Sp, \omega \circ \det, \Sp \otimes \omega \circ \det \} \otimes \overline{\chi} \circ \det$ when $p=3$, we may build the quotient category $\D(\O)[\B] / \T(\O)$. We write $\D(\O)[\B]$ for $\C(\O)[\B]$ for other blocks and thus (\ref{equation:Cblockdecomp}) induces a decomposition of categories
\[
\D(\O) \cong \prod_{\B} \D(\O)[\B].
\]

\subsection{Colmez's Montreal functor}
In \cite{MR2642409}, Colmez has defined an exact and covariant functor $\V$ from the category of smooth, finite-length representations of $G$ on $\O$-torsion modules with a central character to the category of continuous finite-length representations of $\GQp$ on $\O$-torsion modules. If $\chi: \Qp^{\times} \rightarrow \O^{\times}$ is a continuous character, then we may also consider it as a continuous character $\chi: \GQp \rightarrow \O^{\times}$ via class field theory and for all $\pi \in \Modsm(\O)$ of finite length with a central character we have $\V(\pi \otimes \chi \circ \det) \cong \V(\pi) \otimes \chi$. 

Moreover, it follows from the construction in the \textit{loc. cit.} that $\V(\1) = 0$, $\V(\Sp) = \omega$, $\V(\Ind^G_B \chi_1 \otimes \chi_2 \omega^{-1}) \cong \chi_2$, and $\V(\cInd \Sym^r k^2 / (T, S-1)) \cong \ind{\omega_2^{r+1}}$, where $\omega$ is the mod $p$ cyclotomic character, $\omega_2: I_{\Qp} \rightarrow k^{\times}$ is Serre's fundamental character of level $2$, and $\ind{\omega_2^{r+1}}$ is the unique irreducible representation of $\GQp$ of determinant $\omega^r$ and such that $\ind{\omega_2^{r+1}} |_{I_{\Qp}} \cong \omega_2^{r+1} \oplus \omega_2^{p(r+1)}$ with $0 \leq r \leq p-1$. Note that this determined the image of supersingular representations under $\V$ completely since every supersingular representation is isomorphic to $\cInd \Sym^r k^2 / (T, S-1)$ for some $0 \leq r \leq p-1$ after twisting by a character.

Let $\Modpro_{\GQp}(\O)$ be the category of continuous $\GQp$-representations on compact $\O$-modules. Following \cite[\S 3]{MR3306557}, we define an exact covariant functor $\cV: \C(\O) \rightarrow \Modpro_{\GQp}(\O)$ as follows: Let $M$ be in $\C(\O)$, if it is of finite length, we define $\cV(M) := \V(M^\vee)^\vee(\varepsilon \psi)$ where $\vee$ denotes the Pontryagin dual. For general $M \in \C(\O)$, write $M \cong \varprojlim M_i$, with $M_i$ of finite length in $\C(\O)$ and define $\cV(M) := \varprojlim \cV(M_i)$. With this normalization of $\cV$, we have 
\begin{itemize}
\item $\cV(\pi^{\vee}) = 0$ if $\pi \cong \eta \circ \det$;
\item $\cV(\pi^{\vee}) \cong \chi_1$ if $\pi \cong \Ind^G_B \chi_1 \otimes \chi_2 \omega^{-1}$, where $\omega$ is the mod $p$ cyclotomic character;
\item $\cV(\pi^{\vee}) \cong \eta$ if $\pi \cong \Sp \otimes \eta \circ \det$;
\item $\cV(\pi^{\vee}) \cong \V(\pi)$ if $\pi$ is supersingular.
\end{itemize}
The functor $\cV: \C(\O) \rightarrow \Modpro_{\GQp}(\O)$ kills characters and hence every objects in $\I(\O)$. It follows that $\cV$ factors through $\cT: \C(\O) \rightarrow \D(\O)$. We denote $\cV: \D(\O) \rightarrow \Modpro_{\GQp}(\O)$ by the same letter.

If $\Pi \in \Ban(E)$, we abbreviate $\cV(\Pi) = \cV(\Theta^d) \otimes_\O E$ with $\Theta$ is any open bounded $G$-invariant $\O$-lattice in $\Pi$, so that $\cV$ is exact and contravariant on $\Ban(E)$.

\subsection{Extension Computations when $p=3$ and $\B =\{\1, \Sp, \omega \circ \det, \Sp \otimes \circ \det \}$} \label{subsection:extp=3}
In this section we do some similar computations as in \cite[\S 10]{MR3150248} when $p=3$, $\B =\{\1, \Sp, \omega \circ \det, \Sp \otimes \circ \det \}$ with $\omega: \Qp^{\times} \rightarrow k^{\times}$ the character $\omega(x) = x|x|$ mod $\varpi$, and $\psi = \1$. We write $\operatorname{Mod}^{\operatorname{l.fin}}_{G/Z}(k)$ for $\operatorname{Mod}^{\operatorname{l.fin}}_{G, \1}(\O)$ and $e(\pi', \pi) := \dim_k \Ext^1_{G/Z}(\pi', \pi)$ for $\pi , \pi' \in \operatorname{Mod}^{\operatorname{l.fin}}_{G/Z}(k)$. By \cite[Thereom 11.4]{MR2667891}, we have the following table for $e(\pi', \pi)$:\\
\begin{tabular}{c|cccc} 
$\pi' \backslash \pi$ & $\1$ & $\Sp$ & $\omega \circ \det$ & $\Sp \otimes \omega \circ \det$ \\ \hline 
$\1$ & 0 & 2 & 0 & 0 \\ 
$\Sp$ & 1 & 0 & 1 & 0 \\ 
$\omega \circ \det$ & 0 & 0 & 0 & 2 \\ 
$\Sp \otimes \omega \circ \det$ & 1 & 0 & 1 & 0
\end{tabular}\\

Since $e(\1, \Sp) =2$ there exists a unique smooth $k$-representation $\tau$ with socle $\Sp$ and have an exact sequence:
\begin{equation} \label{equation:tau}
0 \rightarrow \Sp \rightarrow \tau \rightarrow \1 \oplus \1 \rightarrow 0.
\end{equation}

\begin{lemma} \label{lemma:extcompu}
$e(\1, \tau) = 0$, $e(\omega \circ \det, \tau) = 0$, $e(\Sp, \tau) = 2$, $e(\Sp \otimes \omega \circ \det, \tau)= 2$.
\end{lemma}

\begin{proof}
Since $e(\1, \1) = 0$, we obtain the first assertion by applying $\Hom_{G/Z}(\1, -)$ to (\ref{equation:tau}). Since $e(\omega \circ \det, \1) = e(\omega \circ \det, \Sp) = 0$, we get the second assertion by applying $\Hom_{G/Z}(\omega \circ \det, -)$ to (\ref{equation:tau}). Note that by applying Emerton's ordinary part functor $\Ord_B$ \cite{MR2667882} to (\ref{equation:tau}), we obtain the exact sequence
\begin{align*}
0 \rightarrow \Ord_B \Sp \rightarrow \Ord_B \tau \rightarrow (\Ord_B \1)^{\oplus 2} \rightarrow \ROrd_B \Sp \rightarrow \ROrd_B \tau \rightarrow (\ROrd_B \1)^{\oplus 2}.
\end{align*}
It follows from \cite[Theorem 4.2.12]{MR2667883} and the fact that $\operatorname{R}^i \Ord_B = 0$ for $i \geq 2$ (c.f. \cite[Proposition 3.6.1]{MR2667883} and \cite{MR2667892}) that we have
\[
\Ord_B \tau \cong \Ord_B \Sp \cong \1, \quad \ROrd \tau \cong (\ROrd \1)^{\oplus 2} \cong (\omega \otimes \omega)^{\oplus 2}
\]
and thus by the 5 terms sequence for $\Ord_B$ \cite[(3.7.5)]{MR2667883},
\[
\Ext^1_{G/Z}(\Ind^G_B \1, \tau) \cong \Ext^1_{T/Z}(\1, \1), \\
\Ext^1_{G/Z}(\Ind^G_B \omega \otimes \omega, \tau) \cong \Hom_{T/Z}(\omega, \omega^{\oplus 2})
\]
are both 2-dimensional. Since $e(\1, \tau) = 0$, by applying $\Hom_{G/Z}(-, \tau)$ to the short exact sequence $0 \rightarrow \1 \rightarrow \Ind^G_B \1 \rightarrow \Sp \rightarrow 0$ we deduce that $\Ext^1_{G/Z}(\Sp, \tau) \cong \Ext^1_{G/Z}(\Ind^G_B \1, \tau)$, which proves the third assertion. Since $e(\omega \circ \det, \tau) = 0$, by applying $\Hom_{G/Z}(-, \tau)$ to the short exact sequence $0 \rightarrow \omega \circ \det \rightarrow \Ind^G_B \omega \otimes \omega \rightarrow \Sp \otimes \omega \circ \det \rightarrow 0$ we deduce that $\Ext^1_{G/Z}(\Sp \otimes \omega \circ \det, \tau) \cong \Ext^1_{G/Z}(\Ind^G_B \omega \otimes \omega, \tau)$, which proves the last assertion. 
\end{proof}

\begin{lemma} \label{lemma:esepi}
If $\Hom_{\C(\O)}(N, \1 \oplus \omega \circ \det) = 0$ then for every essential epimorphism $q: M \twoheadrightarrow N$, $\cT q: \cT M \twoheadrightarrow \cT N$ is an essential epimorphism in $\D(\O)$.
\end{lemma}

\begin{proof}
The proof of \cite[Lemma 10.27, Lemma 10.29]{MR3150248} works verbatim in our setting.
\end{proof}

For an object $M$ of $\C(\O)$, we denote by $I_G(M):= (M^{\vee}/(M^{\vee})^{\SL_2(\Qp)})^{\vee} \subseteq M$.

\begin{lemma} \label{lemma:Dext}
Let $M$ and $N$ be objects of $\C(\O)$. If $I_G(M)=M$ and $N^{\SL_2(\Qp)} = 0$, then the natural map $\Ext^1_{\C(\O)}(M, N) \rightarrow \Ext^1_{\D(\O)}(\cT M, \cT N)$ is an injection.
\end{lemma}

\begin{proof}
Let $P$ be a projective envelope of $M$ in $\C(\O)$. Note that $\cT M$ is projective in $\D(\O)$ and $\Hom_{\C(\O)}(P, N) \cong \Hom_{\D(\O)}(\cT P, \cT N)$ by \cite[Lemma 10.27]{MR3150248}. Consider the commutative diagram of exact sequences
\[
  \begin{tikzcd}
  \Hom_{\C(\O)}(P, N) \arrow{r}{f} \arrow{d}{\sim} &\Hom_{\C(\O)}(K, N) \arrow{r}{g} \arrow[d] &\Ext^1_{\C(\O)}(M, N) \arrow[r] \arrow[d] & 0 \\
  \Hom_{\D(\O)}(\cT P, \cT N) \arrow{r}{\cT f} &\Hom_{\D(\O)}(\cT K, \cT N) \arrow{r}{\cT g} &\Ext^1_{\D(\O)}(\cT M, \cT N) \arrow[r] & 0
  \end{tikzcd}
\]
coming from applying $\Hom_{\C(\O)}(-, N)$ to the short exact sequence $0 \rightarrow K \rightarrow P \rightarrow M \rightarrow 0$ and the functoriality of $\cT$.

We claim that the middle vertical map is injective. Suppose this is the case. For each $c \in \Ext^1_{\C(\O)}(M, N)$ such that $\cT c = 0$, we may find $d \in \Hom_{\C(\O)}(K, N)$ for which $g(d) = c$. Since $\cT g(\cT d) = \cT g(d) = \cT c = 0$, there exists $e \in \Hom_{\C(\O)}(P, N)$ such that $\cT f(\cT e) = \cT d$. It follows from the claim that $f(e) = d$, and thus $c = g \circ f(d) = 0$.

To prove the claim, we consider the exact sequence
\[
0 \rightarrow \Hom_{\C(\O)}(K / I_G(K), N) \rightarrow \Hom_{\C(\O)}(K, N) \rightarrow \Hom_{\C(\O)}(I_G(K), N)
\]
coming from applying $\Hom_{\C(\O)}(-, N)$ to the short exact sequence $0 \rightarrow I_G(K) \rightarrow K \rightarrow K/I_G(K) \rightarrow 0$. By \cite[Lemma 10.26]{MR3150248}, the last term in the exact sequence is isomorphic to $\Hom_{\D(\O)}(\cT K, \cT N)$. Thus it suffices to show that $\Hom_{\C(\O)}(K / I_G(K), N) = 0$, which follows from $(K / I_G(K))^{\SL_2(\Qp)} = K / I_G(K)$ and $N^{\SL_2(\Qp)}=0$.
\end{proof}

Denote $T_\1 := \cT((\Ind^G_B \1)^{\vee})$ and $T_\omega := \cT((\Ind^G_B \omega \otimes \omega)^{\vee})$, both of which are objects in $\D(k)$. Note that since $\cT(\1) = \cT(\omega \circ \det) \cong 0$ in $\D(k)$ and $\cT$ is exact, we have
\begin{align*}
&T_\1 \cong \cT \Sp^{\vee} \cong \cT \tau^{\vee}, &&T_\omega \cong \cT (\Sp \otimes \omega \circ \det)^{\vee}, \\
&\cV(T_\1) \cong \cV(\Sp^\vee) \cong \cV(\tau^{\vee}) \cong \1, &&\cV(T_\omega) \cong \cV((\Sp \otimes \omega \circ \det)^\vee) \cong \omega.
\end{align*}
We denote $\Ext^1_{\D(\O)}(M, N)$ (resp. $\Ext^1_{\D(k)}(M, N)$) the Yoneda extension groups of $M$ by $N$ in $\D(\O)$ (resp. $\D(k)$).

\begin{lemma} \label{lemma:quotextcompu}
$\Ext^1_{\D(k)}(T_\1, T_\1)$, $\Ext^1_{\D(k)}(T_\omega, T_\omega)$, $\Ext^1_{\D(k)}(T_\1, T_\omega)$, and $\Ext^1_{\D(k)}(T_\omega, T_\1)$ are all 2-dimensional.
\end{lemma}

\begin{proof}
By replacing \cite[Lemma 10.12]{MR3150248} with Lemma \ref{lemma:extcompu}, the proof of \cite[Lemma 10.34]{MR3150248} works verbatim in our setting. We include the proof for the sake of completeness. We first note that it suffices to show the assertion for $\Ext^1_{\D(k)}(T_\1, T_\1)$ and $\Ext^1_{\D(k)}(T_\1, T_\omega)$ since $\Ext^1_{\D(k)}(T_\1, T_\1) \cong \Ext^1_{\D(k)}(T_\omega, T_\omega)$ and $\Ext^1_{\D(k)}(T_\1, T_\omega) \cong \Ext^1_{\D(k)}(T_\omega, T_\1)$ by twisting. Let $J_{\Sp}$ (resp. $J_{\Sp \otimes \omega \circ \det}$) be an injective envelope of $\Sp$ (resp. $\Sp \otimes \omega \circ \det$) in $\operatorname{Mod}^{\operatorname{l.fin}}_{G/Z}(k)$. It follows from Lemma \ref{lemma:extcompu} that we have an exact sequence:
\begin{equation} \label{equation:injenv}
0 \rightarrow \tau \rightarrow J_{\Sp} \rightarrow J_{\Sp}^{\oplus 2} \oplus J_{\Sp \otimes \omega \circ \det}^{\oplus 2}.
\end{equation}
Moreover, if we let $\kappa$ be the cokernel of the second arrow then the monomorphism $\kappa \hookrightarrow J_{\Sp}^{\oplus 2} \oplus J_{\Sp \otimes \omega \circ \det}^{\oplus 2}$ induced by the first arrow is essential. Let $\pi$ be $\Sp$ or $\Sp \otimes \omega \circ \det$ then we know from Lemma \ref{lemma:esepi} that $\cT J_\pi^\vee$ is a projective envelope of $\pi^\vee$ in $\D(k)$. By dualizing (\ref{equation:injenv}), applying $\cT$ and then $\Hom_{\D(k)}(-, \cT \pi^\vee)$ we obtain
\[
\Ext^1_{\D(k)}(T_\1, \cT \pi^\vee) \cong \Hom_{\D(k)}(\cT \kappa^\vee, \cT \pi^\vee) \cong \Hom_{\D(k)}(\cT J^{\vee}, \cT \pi^\vee),
\]
where $J = J_{\Sp}^{\oplus 2} \oplus J_{\Sp \otimes \omega \circ \det}^{\oplus 2}$. The last isomorphism follows from the fact that $\cT \pi^\vee$ is irreducible and $\cT J^\vee \twoheadrightarrow \cT \kappa^\vee$ is essential \cite[Lemma 10.29]{MR3150248}. Hence $\Ext^1_{\D(k)}(T_\1, T_\1)$ and $\Ext^1_{\D(k)}(T_\1, T_\omega)$ are 2-dimensional.
\end{proof}

\begin{lemma} \label{lemma:injExtV}
The functor $\cV$ induces an injection
\[
\cV: \Ext^1_{\D(\O)}(S_1, S_2) \hookrightarrow \Ext^1_{\O[\GQp]}(\cV(S_1), \cV(S_2))
\]
for $S_1, S_2 \in \{T_\1, T_\omega\}$.
\end{lemma}

\begin{proof}
Note that Proposition \uppercase\expandafter{\romannumeral7}.4.12, 4.24, and 4.25 in \cite{MR2642409} hold true when $p=3$. Thus the proof of \cite[Lemma 10.35]{MR3150248} works verbatim in our setting with Lemma 10.34 of \textit{loc. cit.} replaced by Lemma \ref{lemma:quotextcompu} above.
\end{proof}

\subsection{A finiteness lemma}
\begin{lemma} \label{lemma:HomVinj}
Let $M, N \in \D(\O)$ be of finite length. Then $\cV$ induces:
\begin{align*}
\Hom_{\D(\O)}(M, N) &\cong \Hom_{\GQp}(\cV(M), \cV(N)), \\
\Ext^1_{\D(\O)}(M, N) &\hookrightarrow \Ext^1_{\GQp}(\cV(M), \cV(N)).
\end{align*}
\end{lemma}

\begin{proof}
This is proved in \cite[Lemma A1]{MR2667891} for supersingular blocks and in \cite{MR3150248} for the remaining except when $p=3$ and $\B =\{\1, \Sp, \omega \circ \det, \Sp \otimes \omega \circ \det\} \otimes \delta \circ \det$, where $\delta: \Qp^{\times} \rightarrow k^{\times}$ is a smooth character. The argument in Pa\v{s}k\={u}nas' proof is by induction on $\ell(M) + \ell(N)$, where $\ell$ denotes the number of irreducible subquotients, and thus reduces the assertion to the case that both $M$ and $N$ are irreducible. Note that in the exceptional case, we may assume that $\delta = 1$ in which case the assertion for $\Hom$ is immediate and the assertion for $\Ext^1$ follows from Lemma \ref{lemma:injExtV}. This proves the lemma.
\end{proof}

Let $\Modpro_{\GQp}(\O)[\B]$ be the full subcategory of $\Modpro_{\GQp}(\O)$ with object $\rho$ such that there exists $M \in \C(\O)[\B]$ such that $\rho \cong \cV(M)$.

\begin{prop} \label{prop:equivB}
The functor $\cV$ induces an equivalence of categories between $\D(\O)[\B]$ and $\Modpro_{\GQp}(\O)[\B]$.
\end{prop}

\begin{proof}
This is due to \cite{MR3150248} except the case that $p=3$ and $\B = \{\1, \Sp, \omega \circ \det, \Sp \otimes \circ \det \}$. Note that in the exceptional case, the proof of \cite[Proposition 10.36]{MR3150248} works verbatim with Lemma 10.35 in the \textit{loc. cit.} replaced by Lemma \ref{lemma:injExtV} above. This proves the proposition.
\end{proof}

\begin{prop} \label{prop:fgadm}
Let $\pi \in \Modlfinc(k)$ be admissible. Then $\cV(\pi^\vee)$ is  finitely generated as $k \bracketGQp$-module.
\end{prop}

\begin{proof}
Without loss of generality, we can assume $\pi \in \Modlfin(k)[\B]$, hence has finitely many irreducible subquotients $\pi_1, ..., \pi_n$ up to isomorphism. Since $\rho_i := \cV(\pi_i^\vee)$ is a finite-dimensional $\GQp$-representation over $k$, $\Ker \rho_i$ has finite index in $\GQp$. It follows that $\K := \bigcap_i \Ker \rho_i$ is of finite index in $\GQp$ and $\H := \GQp / \K$ is a finite group. Denote $\P$ by the maximal pro-$p$ quotient of $\K$ and $\G$ by the quotient of $G$ defined by the diagram
\[
  \begin{tikzcd}
  0 \arrow[r] &\K \arrow[r] \arrow[d, twoheadrightarrow] &\GQp \arrow[r] \arrow[d, twoheadrightarrow] &\H \arrow[r] \ar[equal]{d} &0 \\
  0 \arrow[r] &\P \arrow[r] &\G \arrow[r] &\H \arrow[r] &0 
  \end{tikzcd}
\]

We claim that the action of $\GQp$ on $\cV(\pi^\vee)$ factors through $\G$. Since $\pi$ can be written as an inductive limit of finite length smooth representations, it suffices to prove the claim for $\pi$ of finite length by the definition of $\cV$. Set $\rho = \cV(\pi^\vee)$ a finite-dimensional $\GQp$-representation over $k$. Since $\K$ acts trivially on each irreducible pieces, $\rho(\K)$ has to be contained in the upper triangular unipotent matrices, hence it's a $p$-group. The claim follows because the action of $\K$ factors through $\P$.

We have the following equivalent conditions:
\begin{align*}
\cV(\pi^\vee) \text{ is a finitely generated } k \bracketGQp \text{-module} 
\Longleftrightarrow & \cV(\pi^\vee) \text{ is a finitely generated } k \llbracket \G \rrbracket \text{-module} \\
& (\text{thus a finitely generated } k \llbracket \P \rrbracket \text{-module}) \\ 
\Longleftrightarrow & \cV(\pi^\vee)_\P \text{ is a finitely generated } k[\H] \text{-module} \\
& (\text{thus a finite-dimensional } k \text{-vector space}) \\
\Longleftrightarrow & \text{the cosocle of } \cV(\pi^\vee) \text{ is of finite length}
\end{align*}
where the cosocle is defined to be the maximal semisimple quotient. The first equivalence is due to the claim and the second equivalence follows from Nakayama lemma for compact modules (see \cite[Corollary 1.5]{MR0202790}). The last equivalence follows from the fact that the cosocle of $\cV(\pi^\vee)$ in the category of compact $k \bracketGQp$-modules coincides with the cosocle of $\cV(\pi^\vee)_\P$ in the category of $k[\H]$-modules.

Since $\pi$ is admissible, its pro-$p$ Iwahori fixed part is of finite-dimensional. Thus there are only finitely many irreducible representations in the socle of $\pi$, which implies that cosocle of $\cV(\pi^\vee)$ is of finite length by applying Lemma \ref{lemma:HomVinj} to an inductive limit of finite length smooth representations defining $\pi$. This proves the proposition.
\end{proof}
\section{Patched modules}\label{Section:patched}
From now on we make the assumption $p > 2$. Fix a continuous representation $\rbar: \GQp \rightarrow \GL_2(k)$ and enlarge $k$ if necessary. Corollary A.7 of \cite{MR3134019} (with $K = \Qp$) provides us with an imaginary CM field $F$ with maximal totally real subfield $F^+$, and a continuous representation $\bar{\rho}: G_F \rightarrow \GL_2(k)$ such that $\bar{\rho}$ is a suitable globalization of $\rbar$ in the sense of \S 2.1 of \cite{MR3529394}. 

Let $T = S_p \cup \{v_1\}$, with $S_p$ be the set of places of $F^+$ lying above $p$ and $v_1$ a place prime to $p$ and satisfying the properties in \S 2.3 of \cite{MR3529394}. For each $v \in S_p$, we let $\tilde{v}$ be a choice of a place of $F$ lying over $v$, with the property that $\bar{\rho} \vert_{G_{F_{\tilde{v}}}} \cong \rbar$. (Such a choice is possible by our assumption that $\bar{\rho}$ is a suitable globalization of $\rbar$.) We let $\tilde{T}$ denote the set of places $\tilde{v}$, $v \in T$. For each $v \in T$, we let $R^{\Box}_{\tilde{v}}$ denote the maximal reduced and $p$-torsion free quotient of the universal framed deformation ring of $\bar{\rho} \vert_{G_{F_{\tilde{v}}}}$. We fix a place $\p \in S_p$. For each $v \in S_p \setminus \{\p\}$, we write $\bar{R}^{\Box}_{\tilde{v}}$ for an irreducible, reduced and $p$-torsion free potentially Barsotti-Tate quotient of $R^{\Box}_{\tilde{v}}$ (given by an irreducible component of a potentially crystalline deformation ring of $\rbar$ with Hodge type $(1, 0)$ and some inertial type).

\begin{rmk}
Any $\rbar$ admits a potentially Barsotti-Tate lift \cite[Proposition 7.8.1]{2009arXiv0905.4266S} and any such lift is potentially diagonalizable \cite[Lemma 4.4.1]{MR3292675}.
\end{rmk}

Consider the deformation problem
\begin{gather*}
 \cS = (F/F^+, T, \tilde{T}, \O, \varepsilon^{-1}, \{R^{\Box}_{\tilde{v}_1} \} \cup \{ R^{\Box}_{\tilde{\p}} \} \cup \{\bar{R}^{\Box}_{\tilde{v}}\}_{v \in S_p \setminus \{\p\}}).
\end{gather*}
With this choice of deformation problem and globalization $\bar{\rho}$ for our local Galois representation $\rbar$, the Taylor-Wiles-Kisin patching argument carried out in \S 2 of \cite{MR3529394} produces for some $d > 0$ and $\O[G]$-module $\Minf$ with an arithmetic action of $\Rinf = \Runiv \ctimes_{\O, v \in S_p - \{ \p \} } \bar{R}^{\Box}_{\tilde{v}} \ctimes_{\O} R^{\Box}_{\tilde{v}_1} \bracketx$ (note that $\Runiv \cong R^{\Box}_{\tilde{\p}}$) in the sense of \cite[\S 3]{MR3732208}, which means the $\Rinf[G]$-module $\Minf$ satisfies the following properties:

\begin{enumerate}[label=\textbf{AA\arabic*}]
\item \label{AA1} $\Minf$ is a finitely generated $\Rinf \bracketK$-module.
\item \label{AA2} $\Minf$ is projective in the category of pseudocompact $\O \bracketK$-modules. \\
\item \label{AA3} For $\sigma = \sigma(\lambda, \tau)$ or $\sigma^{cr}(\lambda, \tau)$, we define
\begin{gather*}
\Minf(\sigma^\circ) := \big( \Homc_{\O \bracketK}( \Minf, (\sigma^\circ)^d ) \big)^d \cong \Minf \otimes_{\O \bracketK} \sigma^\circ,
\end{gather*}
where we are considering continuous homomorphism for the profinite topology on $\Minf$ and the $p$-adic topology on $(\sigma^{\circ})^d$. This is a finitely generated $\Rinf$-module by (\ref{AA1}) and corollary 2.5 of \cite{MR3306557}. The action of $\Rinf$ on $\Minf(\sigma^\circ)$ factors through $\Rinf(\sigma) := \Runiv(\sigma) \bracketx$, where $\Runiv(\sigma) = \Runiv(\lambda, \tau)$ (resp. $R_{\rbar}^{\Box, cr}(\lambda, \tau)$) if $\sigma = \sigma(\lambda, \tau)$ (resp. $\sigma^{cr}(\lambda, \tau)$). Furthermore, $\Minf(\sigma^\circ)$ is maximal Cohen-Macaulay over $\Rinf(\sigma)$, and the $\Rinf(\sigma)[1/p]$-module $\Minf(\sigma^\circ)[1/p]$ is locally free of rank 1 over its regular locus.
\item \label{AA4} For such $\sigma$, the action of $\H(\sigma):= \End_{G}(\cInd\sigma)$ on $\Minf(\sigma^\circ)[1/p]$ is given by the composite
\begin{align*}
\H(\sigma) \xrightarrow{\eta} \Runiv(\sigma)[1/p] \rightarrow \Rinf(\sigma)[1/p] 
\end{align*}
where $\H(\sigma) \xrightarrow{\eta} \Runiv(\sigma)[1/p]$ is defined in \cite[Theorem 4.1]{MR3529394} for $\sigma^{cr}(\lambda, \tau)$ and \cite[Theorem 3.6]{2018arXiv180301610P} for $\sigma(\lambda, \tau)$.
\end{enumerate}

\begin{definition}
By \cite[Lemma 4.18 (2)]{MR3529394} and \cite[Proposition 5.5]{2018arXiv180301610P}, the support of $\Minf(\sigma)$ is a union of irreducible components of $\Spec \Runiv(\sigma)$, which we call the set of automorphic components of $\Spec \Runiv(\sigma)$.
\end{definition}

Let $\Piinf := \Homc_\O(\Minf, E)$. If $y \in \mSpec \Rinf[1/p]$, then 
\begin{gather*}
\Pi_y := \Homc_\O(\Minf \otimes_{\Rinf, y} \O_{\kappa(y)}, E) = \Piinf[\m_y]
\end{gather*}
is an admissible unitary $E$-Banach space representation of $G$; see \cite[Proposition 2.13]{MR3529394}. The composition $\Runiv \rightarrow \Rinf \xrightarrow{y} \O$ defines an $\O$-valued point $x \in \Spec \Runiv$ and thus a continuous representation $r_x : \GQp \rightarrow \GL_2(\O)$. We say $y$ is crystalline (resp. semi-stable, resp. potentially crystalline, resp. potentially semi-stable) if $r_x$ is crystalline (resp. semi-stable, resp. potentially crystalline, resp. potentially semi-stable).

If $r: \GQp \rightarrow \GL_2(E)$ is potentially semi-stable of Hodge type $\lambda$, we set 
\begin{gather*}
\BS(r) :=   \pism(r) \otimes \pialg(r),
\end{gather*}
where $\pism(r)$ is the smooth representation of $G$ associated to the Frobenius semi-simplification of $\WD(r)$ via local Langlands correspondence, and $\pialg(r) = \det^b \otimes \Sym^{a-b-1}E^2$ is an algebraic representation of $G$. For $y \in \mSpec \Rinf[1/p]$ potentially crystalline (resp. potentially semi-stable), it is shown in \cite[Theorem 4.35]{MR3529394} (resp. \cite[Theorem 6.1]{2018arXiv180301610P}) that if $\pism(r_x)$ is generic irreducible and $x$ lies on an automorphic component of a potentially crystalline (resp. potentially semi-stable) deformation ring of $\rbar$, then the space of locally algebraic vectors $\Pi_y^{\lalg}$ in $\Pi_y$ is isomorphic to $\BS(r_x)$.

\begin{rmk} \label{rmk:crysBanach}
If $y \in \mSpec \Rinf[1/p]$ lies on an automorphic component and $r_x$ is crystabelline, then $\pism(r_x)$ is a principal series representation and $\BS(r_x)$ admits a unique unitary Banach space completion which is topologically irreducible c.f. \cite[Theorem 4.3.1, Corollary 5.3.1, Corollary 5.3.1]{MR2642406} for $r_x$ absolutely irreducible and \cite[Proposition 2.2.1]{MR2667890} for $r_x$ reducible. Thus the injection $\BS(r_x) \hookrightarrow \Pi_y$ gives rise to an embedding of Banach space representations $\widehat{\BS(r_x)} \hookrightarrow \Pi_y$.
\end{rmk}

\subsection{The action of the center $Z$}\label{section:centralchar}
We now describe the action of the center $Z$ of $G$ on $\Minf$. The determinant of the universal lifting $\runiv$ of $\rbar$ is a character $\det \runiv : \GQp^{ab} \rightarrow (\Runiv)^{\times}$ lifting $\det \rbar$. Hence it factors through $\varepsilon \psiuniv: \GQp \rightarrow \Rpsi^{\times}$, where $\Rpsi$ is the universal deformation ring of $\overline{\psi} = \omega^{-1} \det \rbar$ and $\psiuniv$ is the universal deformation of $\bar{\psi}$. Via pullback along the natural homomorphism $\O[Z] \rightarrow \Rpsi[Z]$, the maximal ideal of $\Rpsi[Z]$ generated by $\varpi$ and the elements $z - \psiuniv \circ \Art(z)$ gives a maximal ideal of $\O[Z]$. If we denote by $\Lambda_Z$ the completion of the group algebra $\O[Z]$ at this maximal ideal, then the character $\psiuniv \circ \Art(z)$ induces a homomorphism $\Lambda_Z \rightarrow \Rpsi$; the corresponding morphism of schemes $\Spec \Rpsi \rightarrow \Spec \Lambda_Z$ associates to each deformation $\psi$ of $\bar{\psi}$ the character $\psi \circ \Art$ of $Z$. Thus $\Lambda_Z$ is identified with $\Rpsi$.

By the local-global compatibility \cite[\S 4.22]{MR3529394}, the action of $\O[Z]$ on $\Minf$ extends to a continuous action of $\Lambda_Z$. Let $z: \Lambda_Z \rightarrow \O$ be an $\O$-algebra homomorphism and $\psi$ be the composition $Z \rightarrow \Lambda_Z^{\times} \xrightarrow{z} \O^{\times}$. We write $\Minfz := \Minf \otimes_{\Lambda_Z, z} \O$, which is an $\Rinfz = \Runivz \ctimes_{\O} \O \bracketx$-module, and $\Piinfz = \Homc_{\O}(\Minfz, E)$.

\begin{prop}
The module $\Minfz$ is an $\O[G]$-module with an arithmetic action of $\Rinfz$. This means that $\Minfz$ satisfies (\ref{AA1}), (\ref{AA2}) and (\ref{AA4}) with $\Rinf$ replaced by $\Rinfz$ and (\ref{AA2}) with pseudocompact $\O \bracketK$-modules replaced by pseudocompact $\O \bracketK$-modules with central character $\psi^{-1}$.
\end{prop}

\begin{proof}
The proof is same as $\Minf$ in \cite{MR3529394}.
\end{proof}

By \cite[Corollary 4.26]{MR3529394}, $\Minfz$ lies in $\C(\O)$.  Thus we may apply Colmez's functor to $\Minfz$ and obtain an $\GQp$-module $\cV(\Minfz)$ with an action of $\Rinfz$, hence an $\Rinfz \llbracket \GQp \rrbracket$-module.

\begin{prop} \label{prop:Minffg}
$\cV(\Minfz)$ is finitely generated over $\Rinfz \bracketGQp$.
\end{prop}

\begin{proof}
Using Nakayama lemma for compact modules, it is enough to show that $\cV(\Minfz) \otimes_{\Rinfz} k \cong \cV(\Minfz \otimes_{\Rinfz} k)$  (See \cite[Lemma 5.50]{MR3150248}) is a finitely generated $k \bracketGQp$-module. Note that $\Minfz \otimes_{\Rinfz} k$ is a finitely generated $k \bracketK$-module by (\ref{AA1}), so its Pontryagin dual is an admissible $K$-representation with a smooth $G$-action, and thus an admissible $G$-representation. The proposition follows from Lemma \ref{prop:fgadm}.
\end{proof}

\subsection{Capture}
Let $Z(K)$ be the center of $K = \GL_2(\Zp)$, $\O \bracketK$ be the completed group algebra, and $\Modpro_K(\O)$ be the category of compact $\O \bracketK$-modules. For a continuous character $\psi: Z(K) \rightarrow \O^{\times}$, we let $\Modpro_{K, \psi}(\O)$ be the full subcategory of $\Modpro_{K}(\O)$ such that $M \in \Modpro_{K, \psi}(\O)$ lies in $\Modpro_{K, \psi}(\O)$ if and only if $Z(K)$ acts on $M$ by $\psi^{-1}$. 

\begin{definition}
Let $\{ V_i \}_{i \in I}$ be a set of continuous $K$-representations on finite-dimensional $E$-vector spaces and let $M \in \Modpro_{K, \psi}(\O)$. We say that $\{ V_i \}_{i \in I}$ captures $M$ if for any proper quotient $M \twoheadrightarrow Q$, we have $\Homc_{\O \bracketK}(M, V_i) \neq \Homc_{\O \bracketK}(Q, V_i)$ for some $i \in I$.
\end{definition}

By Corollary 4.26 of \cite{MR3529394}, $\Minf^{\psi}$ is a nonzero projective object in $\Modpro_{K, \psi}(\O)$, where $\psi = \psi \vert_{Z(K)}$.

\begin{thm} \label{thm:MinfCH}
The action of $\Rinfz \llbracket \GQp \rrbracket$ on $\cV(\Minfz)$ factors through $\Rinfz \bracketGQp / J$, where $J$ is a closed two-sided ideal generated by $g^2-\tr\big(\rinf(g)\big)g+\det\big(\rinf(g)\big)$ for all $g \in \GQp$, where $\rinf : \GQp \rightarrow \GL_2(\Rinfz)$ is the Galois representation lifting $\rbar$ induced by the natural map $\Runivz \rightarrow \Rinfz$.
\end{thm}

\begin{proof}
In \cite[Proposition 2.7]{MR3544298}, it is shown that there is a family of $K$-representations $\{ \sigma_i \}_{i  \in I}$, where $\sigma_i$ is a  type for a Bernstein component containing a principal series representation but not a special series tensoring with an algebraic representation, which captures every projective object of $\Modpro_{K, \psi}(\O)$. We have the following commuting diagram:
\[
  \begin{tikzcd}
  &\bigoplus_{i \in I} \Hom_K(\sigma_i, \Piinfz) \otimes \sigma_i  \arrow[r] &\Piinfz \\
  &\bigoplus_{i \in I} \bigoplus_{\ y \in \mSpec \Rinfz(\sigma_i)[1/p]} \Hom_K(\sigma_i, \Piinfz[\m_y]) \otimes \sigma_i  \arrow[r, "(\star)"] \arrow[u, hookrightarrow] & \bigoplus_{i \in I} \bigoplus_{y \in \mSpec \Rinfz(\sigma_i)[1/p]} \Piinfz[\m_y] \arrow[u, hookrightarrow, "(*)"].
  \end{tikzcd}
\]
Since $\H(\sigma_i)$ acts on $\Hom_K(\sigma_i, \Piinfz)$ via $\H(\sigma_i) \rightarrow \Runivz(\sigma_i)[1/p] \rightarrow \Rinfz(\sigma_i)[1/p]$ by (\ref{AA4}), it acts on $\Hom_K(\sigma_i, \Piinfz[\m_y])$ via $\H(\sigma_i) \rightarrow \Runivz(\sigma_i)[1/p] \rightarrow \Rinfz(\sigma_i)[1/p] \xrightarrow{y} \kappa(y)$. Thus by applying the Frobenius reciprocity to $(\star)$, we obtain a map
\begin{gather*}
\Hom_G( \cInd \sigma_i \otimes_{\H(\sigma_i), y} \kappa(y), \Piinfz[\m_y]) \otimes (\cInd \sigma_i \otimes_{\H(\sigma_i), y} \kappa(y) ) \rightarrow \Piinfz[\m_y].
\end{gather*}
Since $\Piinfz[\m_y] = \Homc_G(\Minfz \otimes_{\Rinfz, y} \kappa(y), E)$, the image of this map is nonzero if and only if $y$ lies in the support of $\Minfz(\sigma)$, which implies $x$ is potentially crystalline of type $\sigma_i$ (recall $x \in \mSpec \Runivz$ is the point induced by $y$). We can also deduce from (\ref{AA3}) that the dimension of $\Hom_G( \cInd \sigma_i \otimes_{\H(\sigma_i), y} \kappa(y), \Piinfz[\m_y])$ over $\kappa(y)$ is 1 if $y$ lies in the support of $\Minfz(\sigma_i)[1/p]$ and 0 otherwise. In case it is nonzero, we have $\cInd \sigma_i \otimes_{\H(\sigma_i), y} \kappa(y) = \BS(r_x)$ by \cite[Proposition 3.3]{MR1955206}, and thus induces an injection $\widehat{\BS(r_x)} \hookrightarrow \Piinfz[\m_y] \hookrightarrow \Piinfz$ by Remark \ref{rmk:crysBanach}. Let $\widehat{\BS(r_x)}^{\circ} := \widehat{\BS(r_x)} \cap (\Minfz)^d$ be a $G$-invariant $\O$-lattice of $\widehat{\BS(r_x)}$. We define $N$ to be the kernel of the $\Rinfz$-algebra homomorphism
\begin{gather}
\Minfz \rightarrow \prod_{i \in I} \prod_{y} \big( \widehat{\BS(r_x)}^{\circ} \big)^d \label{equ:1}
\end{gather}
induced by ($*$) and $\widehat{\BS(r_x)} \hookrightarrow \Piinfz[\m_y] \hookrightarrow \Piinfz$, where $y \in \mSpec \Rinfz(\sigma_i)[1/p]$ lies in the support of $\Minfz(\sigma_i^{\circ})[1/p]$, and define $M$ by the exact sequence
\begin{gather}
0 \rightarrow N \rightarrow \Minfz \rightarrow M \rightarrow 0 \label{equ:2}
\end{gather}
in $\C(\O)$ with a compatible action of $\Rinfz$.

Tensoring (\ref{equ:2}) with $\sigma_i^{\circ}$ over $\O \bracketK$, we obtain a surjection $\Minfz(\sigma_i^{\circ}) \twoheadrightarrow M(\sigma_i^{\circ})$. By the definition of $M$, we see that $\widehat{\BS(r_x)} \hookrightarrow M^d$ if $y$ lies in the support of $\Minfz(\sigma_i^{\circ})[1/p]$ for some $i \in I$, thus $M(\sigma_i^{\circ})[1/p]$ is supported at each point of $\Rinfz(\sigma_i^{\circ})[1/p]$ at which $\Minfz(\sigma_i^{\circ})[1/p]$ is supported. Since $\Minfz(\sigma_i^{\circ})[1/p]$ is locally free of rank one over its support and $\Rinfz(\sigma_i^{\circ})$ is $p$-torsion free, we deduce that $M(\sigma_i^{\circ}) \cong \Minfz(\sigma_i^{\circ})$ for all $i$, which implies that $M = \Minfz$ by capture and thus $N=0$. Note that $\C(\O)$ is abelian and closed under products. Thus the target of (4) lies in the domain of $\cV$ and we have an injection
\begin{gather}
\cV(\Minfz) \cong \cV(M)  \hookrightarrow \prod_{i \in I} \prod_{y} \cV \Big( \big( \widehat{\BS(r_x)}^{\circ} \big)^d \Big). \label{equ:3}
\end{gather}

We claim that the action of $\Rinfz \bracketGQp$ on $\cV\big(( \widehat{\BS(r_x)}^{\circ} )^d \big)$ factors through $\O_{\kappa(x)} \bracketGQp / J_x$, where $J_x$ is the closed two-sided ideal generated by $g^2 - \tr(r_x(g)) g + \det(r_x(g))$ for all $g \in \GQp$.  Given the claim, we see that $g^2- \tr(\rinf(g))g+ \det(\rinf(g))$ acts by 0 on the right hand side of (\ref{equ:3}), and thus on $\cV(\Minfz)$. This proves the proposition.

To prove the claim, we note that $\BS(r_x)$ is the locally algebraic vectors of the unitary Banach representation $B(r_x)$ constructed in \cite{MR2642406}. Hence we have $r_x \cong \cV(B(r_x)) \twoheadrightarrow \cV(\widehat{\BS(r_x)})$, and the claim follows.
\end{proof}

\begin{corollary} \label{cor:VMfinite}
$\cV(\Minfz)$ is finitely generated over $\Rinfz$.
\end{corollary}

\begin{proof}
By Proposition \ref{prop:Minffg} and Theorem \ref{thm:MinfCH}, $\cV(\Minfz)$ is a finitely generated $\Rinfz \bracketGQp / J$-module, so it suffices to show that $\Rinfz \bracketGQp / J$ is finitely generated over $\Rinfz$. We note that $\Rinfz \bracketGQp / J$ is a Cayley-Hamilton algebra with residual pseudorepresentation associated to $\rbar$ in the sense of \cite[\S 2.2]{MR3831282}, hence it is finitely generated over $\Rinfz$ by \cite[Proposition 3.6]{MR3831282}.
\end{proof}

\begin{prop} \label{prop:Rinffaithful}
$\Rinfz$ acts on $\cV(\Minfz)$ faithfully.
\end{prop}

\begin{proof}
Identify $\Lambda_Z$ with the universal deformation ring of the trivial character, we obtain an isomorphism $\Runiv \cong \Runivz \ctimes_{\O} \Lambda_Z$ via $(r, \chi) \mapsto r \otimes \chi^{1/2}$ \cite[\S 6.1]{MR3732208}, where $\chi^{1/2}$ is a square root of $\chi$ lifting $\1$, and thus an isomorphism $\Rinf \cong \Rinfz \ctimes_{\O} \Lambda_Z$. Consider the $\Rinf$-module $M:= \Minfz \ctimes_{\O} \Lambda_Z$. Note that $M$ carries an arithmetic action of $\Rinf$ by the proof of \cite[Proposition 6.10, Proposition 6.14, Proposition 6.17]{MR3732208}. By applying \cite[Theorem 6.3]{MR4077579} to $M$, we see that $\Rinf$ acts faithfully on $M$ since $\Spec \Runiv$ is irreducible and reduced (formally smooth except for $\rbar \sim (\begin{smallmatrix} 
\1 & *    \\ 0 & \omega 
\end{smallmatrix} 
) \otimes \chi$ or $p=3$ and $\rbar \vert_{I_{\Qp}} \sim (\begin{smallmatrix} 
\omega_2^2 & *    \\ 0 & \omega_2^6 
\end{smallmatrix}
) \otimes \chi$, where $\omega_2$ is a fundamental character of level 2; see \cite[Corollary B.5]{MR3150248} and \cite[Theorem 4.2, Theorem 5.2]{MR2642411}).

Consider $V := \cV(\Minfz) \ctimes_{\O} \1^{univ}$. Note that each point $y \in \Spec \Rinf$ gives a pair $(w, z)$ with $w \in \Spec \Runivz$ and $z \in \Spec \Lambda_Z$ via $\Rinf \cong \Rinfz \ctimes_{\O} \Lambda_Z$, which satisfies 
\[M \otimes_{\Rinf, y} \kappa(y) \cong \big(\Minfz \otimes_{\Rinfz, w} \kappa(w) \big) \otimes \psi_z.
\] 
For $y \in \mSpec \Rinf[1/p]$ crystabelline with a principal series type, we have $\widehat{\BS(r_x)} \hookrightarrow \Pi_y$ by Remark \ref{rmk:crysBanach} and thus $0 \neq \cV(\widehat{\BS(r_x)}) \hookrightarrow \cV(\Pi_y) \neq 0$. Since such points are dense in the support of $M$ \cite[Theorem 5.1]{MR4077579} and $M$ is faithful over $\Rinf$, we deduce that $V$ is faithful as $\Rinf$-module. This implies that $\cV(\Minfz)$ is faithful as $\Rinfz$-module and the proposition follows.
\end{proof}

\begin{corollary} \label{cor:nonvanishing}
For all $y \in \mSpec \Rinf[1/p]$, we have $\cV(\Pi_y) \neq 0$. In particular, $\Pi_y \neq 0$.
\end{corollary}

\begin{proof}
Let $\psi \varepsilon$ be the character given by $\det r_x$. Since $\Rinfz$ acts on $\cV(\Minfz)$ faithfully by Proposition \ref{prop:Rinffaithful} and $\cV(\Minfz)$ is finitely generated over $\Rinfz$ by Corollary \ref{cor:VMfinite}, Nakayama's lemma implies that $\cV(\Minfz) \otimes_{\Rinfz, y} \kappa(y) \neq 0$. On the other hand, since $\cV(\Pi_y) \cong \cV(\Minfz \otimes_{\Rinfz, y} \kappa(y)) \cong \cV(\Minfz) \otimes_{\Rinfz, y} \kappa(y)$, it follows that $\cV(\Pi_y) \neq 0$.
\end{proof}
\section{Main results} \label{Sectiuon:main}
\begin{thm} \label{thm:irrauto}
For $y \in \mSpec \Rinf[1/p]$ whose associated Galois representation $r_x$ is absolutely irreducible, we have $\cV(\Pi_y) \cong r_x^{\oplus n_y}$ for some integer $n_y \geq 1$. In particular, $\Minf(\sigma^{\circ})[1/p]$ is supported on every non-ordinary (at $\p$) component of $\Rinf(\sigma)[1/p]$ for each locally algebraic type $\sigma$ for $G$.
\end{thm}

\begin{proof}
Let $x \in \mSpec \Runiv[1/p]$ be the image of $y \in \mSpec \Rinf[1/p]$ and let $r_x: \GQp \rightarrow \GL_2(\kappa(y))$ be the corresponding Galois representation. If $r_x$ is absolutely irreducible, then the action of $\Rinf \bracketGQp \otimes_{\Rinf, y} \kappa(y)$ on $\cV(\Pi_y)$ factors through $g^2- \tr(r_x(g))g+ \det(r_x(g))$ by Theorem \ref{thm:MinfCH}, which implies $\cV(\Pi_y) \cong (r_x)^{\oplus n_y}$ (see \cite[Theorem 1]{MR1094193}) for some integer $n_y$. As $\cV(\Pi_y) \neq 0$ by Corollary \ref{cor:nonvanishing}, we get $n_y \geq 1$, which proves the first assertion.

Suppose furthermore that $r_x$ is potentially semi-stable of type $\sigma$ and $\pism(r_x)$ is generic. Since $\Pi_y$ is an admissible Banach space representation of $G$, it contains an irreducible subrepresentation $\Pi$ by \cite[Lemma 5.8]{MR2608966}, which can be assumed to be absolutely irreducible after extending scalars. By our assumptions and Lemma \ref{Lemma:SL2inv} below, we have $\Pi^{\SL_2(\Qp)} = 0$. Thus $\cV(\Pi)$ is nonzero and $\dim_E \cV(\Pi) \leq 2$ by \cite[Corollary 1.7]{MR3272011}. Since $(r_x)^{\oplus n_y} \cong \cV(\Pi_y) \twoheadrightarrow \cV(\Pi)$, we deduce $\cV(\Pi) \cong r_x$. Moreover, it follows from \cite[Theorem \uppercase\expandafter{\romannumeral6}.6.50]{MR2642409} and \cite[Theorem 1.3]{MR3272011} that we have $\Pi^{\lalg} \cong \pi_{\lalg}(r_x)$, which implies $y$ lies in the support of $\Minf(\sigma^{\circ})$ since
\[
1 = \dim_{E} \Hom_K(\sigma, \Pi) \leq \dim_E \Hom_K(\sigma, \Pi_y).
\]
Thus the second assertion is a consequence of the fact that a potentially semi-stable deformation ring is equidimensional after inverting $p$ and its non-generic locus (i.e. $\pism(r_x)$ non-generic) has positive codimension \cite[Theorem 1.2.7]{MR3546966}.
\end{proof}

\begin{lemma} \label{Lemma:SL2inv}
Let $y \in \mSpec \Rinf[1/p]$ and $\Pi$ be an absolutely irreducible subrepresentation of $\Pi_y$. If there exists no character $\chi: G_{\Qp} \rightarrow 1 + \varpi \O$, such that $r_x \otimes \chi$ is potentially crystalline of Hodge type $(0,1)$, inertial type $\eta \oplus \eta$ and $\pism(r_x \otimes \chi)$ is non-generic, then $\Pi^{\SL_2(\Qp)} = 0$.
\end{lemma}

\begin{proof}
Suppose $\Pi^{\SL_2(\Qp)} \neq 0$. Then we have $\Pi = \Pi^{\SL_2(\Qp)} \cong \Psi \circ \det$ for some continuous character $\Psi: \Qp^{\times} \rightarrow \kappa(y)^{\times}$ since $\Pi$ is absolutely irreducible. Write $\det r_x = \psi \varepsilon$. Then by identifying $\psi$ with character of $\Qp^\times$ via $\Art$, we see that $\psi = \Psi^2$. Since $p>2$, we may find a character $\chi: \Qp^{\times} \rightarrow 1 + \varpi \O$ such that $\chi^2 \psi |_{1 + p \Zp} = 1$ by Hensel's lemma.

Let $\eta: \Qp^{\times} \rightarrow \kappa(y)^{\times}$ be the smooth character defined by $x \mapsto \Psi(x)\chi(x)$. Then $\chi^2 \det r_x = \varepsilon (\chi^2 \psi) = \varepsilon \eta^2$ is de Rham of Hodge-Tate weight $1$ and we have an injection
\begin{align*}
0 \neq \Hom_K(\eta \circ \det, \Pi \otimes \chi \circ \det) \hookrightarrow \Hom_K(\eta \circ \det, \Pi_y \otimes \chi \circ \det).
\end{align*}
Since $M_{\chi} := \Minf \otimes \chi \circ \det$ carries an arithmetic action of $\Rinf$ (with $G$-action twisted by $\chi \circ \det$), we see that $y$ lies in the support of $M_{\chi}(\eta \circ \det)$ and thus $r_x \otimes \chi$ is potentially crystalline of Hodge type $\lambda= (1, 0)$ and inertia type $\tau = \eta \oplus \eta$. 

Suppose further that $\pism(r_x \otimes \chi)$ is generic. Then $\pism(r_x \otimes \chi)$ is an irreducible principal series and $\Pi' := \widehat{\BS(r_x \otimes \chi)}$ is a subrepresentation of $\Pi_y \otimes \chi$ \cite[Theorem 5.3]{MR3529394}. Note that $\Pi \otimes \chi \circ \det$ is not a subrepresentation of $\Pi'$ since $\Pi'$ is absolutely irreducible and not a character (see \cite[Theorem 4.3.1]{MR2642406} for $r_x \otimes \chi$ absolutely irreducible and \cite[Proposition 2.2.1]{MR2667890} for $r_x \otimes \chi$ reducible). Thus we have an injection
\begin{align*}
0 \neq \Hom_K(\eta \circ \det, \Pi \otimes \chi \circ \det \oplus \Pi') \hookrightarrow \Hom_K(\eta \circ \det, \Pi_y \otimes \chi \circ \det).
\end{align*}
This implies that $M_{\chi}(\eta \circ \det) \otimes_{\Rinf} \kappa(y)$ has dimension greater or equal to 2, which contradicts \ref{AA3} since $M_{\chi}(\eta \circ \det)[1/p]$ is locally free of rank one over $\Rinf(\eta \circ \det)[1/p]$. Therefore, $\pism(r_x \otimes \chi)$ is non-generic and the lemma is proved.
\end{proof}

\begin{corollary}
Let $y \in \mSpec \Rinf[1/p]$. If $r_x$ is absolutely irreducible and potentially semi-stable with $\pism(r_x)$ generic, then we have $\cV(\Pi_y) \cong r_x$. In particular, $\cV(\Pi_y) \cong r_x$ in an open dense subset of $\mSpec \Rinf[1/p]$.
\end{corollary}

\begin{proof}
Let $\Pi$ be an absolutely irreducible subrepresentation of $\Pi_y$ (after extending scalars). Since $\Pi^{\SL_2(\Qp)} = 0$ by Lemma \ref{Lemma:SL2inv}, we see that $\cV(\Pi) \cong r_x$. Let $\Pi^{\circ}_y$ and $\Pi^{\circ}$ be $G$-invariant $\O$-lattices of $\Pi_y$ and $\Pi$ respectively. Then we have
\begin{align*}
    \Homc_G(\Pi, \Pi_y) 
    &= \Hom_{\C(\O)}((\Pi^{\circ}_y)^d, (\Pi^{\circ})^d) \otimes_{\O} E \\
    &= \Hom_{\D(\O)}(\cT (\Pi^{\circ}_y)^d, \cT (\Pi^{\circ})^d) \otimes_{\O} E \\
    &= \Hom_{\GQp}(\cV(\Pi_y), \cV(\Pi)) \\
    &= \Hom_{\GQp}(r_x^{\oplus n_y}, r_x),
\end{align*}
where the second equality follows from \cite[Lemma 10.26]{MR3150248} and the fact $\Pi_y^{\SL_2(\Qp)} = \Pi^{\SL_2(\Qp)} = 0$ and the third equality follows from Proposition \ref{prop:equivB} and \cite[Theorem 3.14 (i)]{MR3272011}. This implies that $\dim_E \Homc_G(\Pi, \Pi_y) = n_y$ and we have an inclusion $\Pi^{\oplus n_y} \subset \Pi_y$. Therefore, the first assertion follows immediately from
\[
1 = \dim_E \Hom_K(\sigma, \Pi_y) \geq \dim_E \Hom_K(\sigma, \Pi^{\oplus n_y}) = n_y \cdot \dim_E \Hom_K(\sigma, \Pi) = n_y,
\]
where the first equality is due to \ref{AA3} since $y$ lies in the regular locus of $\Rinf(\sigma)$ by our assumption that $\pism(r_x)$ is generic \cite[Theorem 1.2.7]{MR3546966}).

To show the second assertion, note that the set of crystabelline points with a fixed Hodge type $\lambda$ are dense in the support of $\Minf$ \cite[Theorem 5.3]{MR4077579} and the reducible locus of $\Spec \Runiv$ is a subset of positive codimension. Thus $\cV(\Pi_y) \cong r_x$ in a dense subset of $\Spec \Rinf[1/p]$ by the first assertion. Using the fact that $y \mapsto \dim_{\kappa(y)} \cV(\Pi_y)$ is an upper semicontinuous function, we deduce that $\cV(\Pi_y) \cong r_x$ in an open dense subset of $\mSpec \Rinf[1/p]$.
\end{proof}

\begin{prop} \label{prop:autocomp}
For $\sigma = \sigma(\lambda, \tau)$ or $\sigma^{cr}(\lambda, \tau)$, every irreducible component of $\Runiv(\sigma)[1/p]$ is automorphic, except possibly for the semi-stable non-crystalline component in the case $\rbar \sim ( \begin{smallmatrix} 
\omega & * \\ 0 & \mathbbm{\1} 
\end{smallmatrix} ) \otimes \chi$, $\lambda=(a+1,a)$ and $\tau = \eta \oplus \eta$.
\end{prop}

\begin{proof}
By Theorem \ref{thm:irrauto}, we only have to handle components in the reducible locus. If $y \in \Spec \Rinfz(\sigma)[1/p]$ is reducible, then $r_x$ is a de Rham representation which is an extension of de Rham characters $\chi_2$ by $\chi_1$. Write $V = \chi_1 \chi^{-1}_2$, which is a de Rham character, hence of the form $\varepsilon^k \mu$, where $k$ is an integer and $\mu$ is a character such that $\mu \vert_{\IQp}$ has a finite image. By \cite[Proposition 1.24]{MR1263527}, we may compute the dimension of crystalline (resp. de Rham) extensions of $\Qp$ by $V$ using the following formulas:
\begin{align*}
h^1_f(V) 
&= h^0(V) + \dim_{\Qp} \DdR(V)/\Fil^0 \DdR(V)\\
&= \begin{cases}
	 1&  \text{if } k < 0 \text{ or }  (k, \mu) = (0, \mathbbm{1})\\
	 0&  \text{otherwise}
   \end{cases} \\
h^1_g(V) 
&= h^1_f(V)  +  \dim_{\Qp}\Dcris(V^*(1))^{\varphi=1} \\
&= \begin{cases}
	 2&  \text{if } (k, \mu) = (-1, \mathbbm{1}) \\
	 1&  \text{if } k < 0, \ (k, \mu) \neq (-1, \mathbbm{1}) \text{ or }  (k, \mu) = (0, \mathbbm{1})\\
	 0&  \text{otherwise}
   \end{cases}
\end{align*}

This shows that for all such $y$, $r_x$ is potentially crystalline ordinary (and thus potentially diagonalizable) in the sense of \cite[\S 1.4]{MR3152941} except the case specified. By applying Theorem A.4.1 of \cite{MR3072811} (see also \cite[Corollary 5.4]{MR3529394}), we may construct a global automorphic Galois representation corresponding to a point on the same component as $y$. This implies that all such components are automorphic, which completes the proof.
\end{proof}

\begin{rmk} \label{rmk:pstncry}
In the case $\rbar \sim ( \begin{smallmatrix} 
\omega & *    \\ 0 & \mathbbm{\1} 
\end{smallmatrix} ) \otimes \chi$, $\lambda=(a+1,a)$ and $\tau = \eta \oplus \eta$, then $\cZ\big(\Runiv(\sigma(\lambda, \tau)) / \omega \big)$ is irreducible if $\rbar$ is tr\'es ramifi\'e, and is the sum of two irreducible components if $\rbar$ is peu ramifi\'e (include split)  with one of which is $\cZ\big(\Runiv(\sigma^{cr}(\lambda, \tau)) / \omega \big)$ and the other of which is the closure of the semi-stable non-crystalline points; see Proposition 3.3.1 of \cite{MR3134019}.
\end{rmk}

\begin{proof}[Proof of Theorem \ref{thm:Breuil-Mezard}]
We follow the strategy in \cite{MR3134019} and \cite{MR3306557}. For every smooth irreducible $k$-representation $\bsigma$ of $K$, we define $\C_{\bsigma}(\rbar)$ to be the cycle $\cZ_{4}\big(\Runiv(\tilde{\sigma}) / \varpi \big)$, where $\tilde{\sigma}$ is a lift of $\bsigma$ of the form $\sigma^{cr}(\lambda, \1 \oplus \1)$ (see Remark \ref{rmk:cryslift}). Assuming $\sigma$ as in Proposition \ref{prop:autocomp}, we have the following equality of cycles:
\begin{align*}
\cZ \big(\Runiv(\sigma) / \varpi \big) \times \cZ(k \bracketx)
&= \cZ \big(\Rinf(\sigma) / \varpi \big) \\
&= \cZ \big(\Minf(\sigma^{\circ})/ \varpi \big) \\
&= \sum_{\bsigma} m_{\bsigma}(\lambda, \tau) \cZ \big(\Minf(\bsigma)\big) \\
&= \sum_{\bsigma} m_{\bsigma}(\lambda, \tau) \cZ \big(\Rinf(\tilde{\sigma}) / \varpi \big) \\
&= \sum_{\bsigma} m_{\bsigma}(\lambda, \tau) \C_{\bsigma}(\rbar) \times \cZ(k \bracketx),
\end{align*}
where $\Minf(\bsigma)$ is a $\Rinf(\tilde{\sigma}) / \varpi$-module defined by $\Homc_{\O \bracketK}( \Minf, \bsigma^\vee)^\vee \cong \Minf(\tilde{\sigma}) / \varpi \Minf(\tilde{\sigma})$. Note that the first and the last equalities follow from the definition of $\Rinf(\sigma)$, the second and fourth equalities follows from Proposition \ref{prop:autocomp} and \cite[Lemma 5.5.1]{MR3134019}, and the third equality follows from the projectivity of $\Minf$; see \cite[Lemma 2.3]{MR3306557}. This proves the theorem but the exceptional case.

For the exceptional case (see Remark \ref{rmk:pstncry}), we may check the conjecture by comparing this cycle with $\cZ\big(\Runiv(\tilde{\sigma}) / \omega \big)$ with $\bar{\sigma} = \overline{\tilde{st} \otimes \eta}$ directly; see \cite[Theorem 5.3.1]{MR1944572} and \cite[\S 3.2.7]{MR2551764}.
\end{proof}

\begin{corollary}
Every irreducible component of $\Runiv(\sigma)[1/p]$ is automorphic.
\end{corollary}

\begin{proof}
Given Theorem \ref{thm:Breuil-Mezard}, this follows from Theorem 5.5.2 of \cite{MR3134019}.
\end{proof}

Following from the method of \cite[Theorem 2.2.18]{MR2505297} and \cite[Theorem 6.3]{MR3429471}, we obtain the Fontaine-Mazur conjecture below.

\begin{thm}
Let $p$ be an odd prime and let $F$ be a totally real field in which $p$ splits completely. Let $\rho : G_{F, S} \rightarrow \GL_2(\O)$ be a continuous representation. Suppose that
\begin{enumerate}
\item $\bar{\rho}$ is modular (thus $\rho$ and $\bar{\rho}$ are odd).
\item $\bar{\rho} \vert_{F(\zeta_p)}$ is absolutely irreducible.
\item $\rho \vert_{G_{F_v}}$ is potentially semi-stable with distinct Hodge-Tate weights for each $v | p$.
\end{enumerate}
Then (up to twist) $\rho$ comes from a Hilbert modular form.
\end{thm}

\begin{rmk}
The result is new when $p=3$ and $\bar{\rho} |_{G_{F_v}}$ is a twist of an extension $\1$ by $\omega$ for some $v | p$.
\end{rmk}

\bibliographystyle{alpha}
\bibliography{Reference}

\end{document}